\documentclass[reqno]{amsart}

\usepackage{amssymb,amsfonts,amstext,amsthm,hyperref,cleveref,xcolor}

\usepackage{graphics}
\usepackage{graphicx}
\usepackage{epstopdf}% To incorporate .eps illustrations using PDFLaTeX, etc.
\usepackage{indentfirst}
\usepackage{float}
\usepackage[numbers,sort&compress]{natbib}% Citation support using natbib.sty
\bibpunct[, ]{[}{]}{,}{n}{,}{,}% Citation support using natbib.sty
\makeatletter% @ becomes a letter
\def\NAT@def@citea{\def\@citea{\NAT@separator}}% Suppress spaces between citations using natbib.sty
\makeatother% @ becomes a symbol again
\theoremstyle{plain}% Theorem-like structures provided by amsthm.sty
\newtheorem{theorem}{Theorem}[section]
\newtheorem{lemma}[theorem]{Lemma}
\newtheorem{corollary}[theorem]{Corollary}

\crefrangeformat{equation}{#3(#1)#4--#5(#2)#6}
\crefname{enumi}{\unskip}{\unskip}

\theoremstyle{definition}

\newtheorem{remark}[theorem]{Remark}

\begin{document}

\title[Expressing Matrices into products of commutators]{Expressing Matrices into products of commutators of involutions, skew-involutions, finite order and skew finite order matrices}

\author{Ivan Gargate}
\address{UTFPR, Campus Pato Branco, Rua Via do Conhecimento km 01, 85503-390 Pato Branco, PR, Brazil}
\email{ivangargate@utfpr.edu.br}

\author{Michael Gargate}
\address{UTFPR, Campus Pato Branco, Rua Via do Conhecimento km 01, 85503-390 Pato Branco, PR, Brazil}
\email{michaelgargate@utfpr.edu.br}

\begin{abstract}
Let $R$ be an associative ring with unity $1$ and consider that $2,k$ and $2k\in \mathbb{N}$ are invertible in $R$. For $m\geq 1$ denote by $UT_n(m,R)$ and $UT_{\infty}(m,R)$, the subgroups of $UT_n(R)$ and $UT_{\infty}(R)$ respectively, which have zero entries on the first $m-1$ super diagonals. We show that every element on the groups $UT_n(m,R)$ and  $UT_{\infty}(m,R)$ can be expressed as a product of two commutators of involutions and also, can be expressed as a product of two commutators of skew-involutions and involutions in $UT_{\infty}(m,R)$. Similarly, denote by $UT^{(s)}_{\infty}(R)$  the group of upper triangular infinite matrices whose diagonal entries are $s$th roots of $1$. We show that every element of the groups $UT_n(\infty,R)$ and $UT_{\infty}(m,R)$ can be expressed as a product of $4k-6$ commutators all depending of powers of elements in $UT^{(k)}_{\infty}(m,R)$ of order $k$ and, also, can be expressed as a product of $8k-6$ commutators of skew finite matrices of order $2k$ and matrices of order $2k$ in $UT^{(2k)}_{\infty}(m,R)$. If $R$ is the complex field or the real number field we prove that, in $SL_n(R)$ and in the subgroup $SL_{VK}(\infty,R)$ of
the Vershik-Kerov group over $R$, each element in these groups can be decomposed into a product of  commutators of elements as described above.
\end{abstract}

%\subjclass[2010]{Primary 16S50, 17C50; Secondary 16W10}

\keywords{Upper triangular matrices; finite order; commutators }

\maketitle

\section{Introduction}\label{intro}
It is a classical question whether the elements of a ring or a group can be expressed as sums or products of elements of some particular set. For example, expressing matrices as a product of involutions was studied by several authors \cite{Slowik-1, Djocovi,Halmos, Halmos2}. In case of product of comutators  we can see \cite{Zheng, Hou,Paras}. Also in \cite{Paras}  the author shows the necessary and sufficient condition for a matrix over a field to be the product of an involution  and a skew-involution.

Bier and Waldemar in \cite{Bier} studied the commutators of elements of the group $UT_{\infty}(m,R)$ of infinite unitriangular matrices over an associative ring $R$ with unity $1$ containing exactly all these matrices,
which have zero entries on the first $m-1$ superdiagonals. They prove that every unitriangular matrix of a specified form is a commutator of two other unitriangular matrices. Considering $\omega$ a kth  root of unity  in $R$, recently Gargate in \cite{Gargate} prove that every element of the
group $UT_{\infty}(R)$ can be expressed as a product of $4k-6$ commutators all depending of powers of elements in $UT_{\infty}^{(k)}(R)$ of order $k$. 

In the section 3 following the same direction,  we study the subgroup $UT_n( m,R)$ and $UT_{\infty}(m,R)$ of $UT_n(R)$ and $UT_{\infty}(m,R)$ respectively which have zero entries on the first $m-1$ superdiagonals.

The main result of this paper is stated as follows:

\begin{theorem}\label{th1}
 Let $R$ be an associative ring with unity $1$ and suppose that $2$,$k$ and $2k$ are  invertible in $R$. Then every matrix in $UT_{\infty}(m,R)$ and $UT_n(m,R)$ can be expressed as a product of at most:
 \begin{itemize}
     \item [1.)] Two commutators of involutions in $\pm UT_{\infty}^{(2)}(m,R)$.
     \item[2.)] Two commutators of  skew-involutions and  involutions both in $\pm UT_{\infty}^{(2)}(m,R)$.
     \item[3.)] $4k-6$ commutators of matrices of order $k$ in $UT_{\infty}^{(k)}(m,R)$.
     \item[4.)] $8k-6$ commutators of matrices of skew order $2k$ and matrices of order finite $2k$ in $\pm UT_{\infty}^{(2k)}(m,R)$. 
 \end{itemize}
 
 \end{theorem}

 Paras and Salinasan in \cite{Paras} shows the necessary and sufficient condition for a matrix $A$ over a field to be the product of an involution and a skew-involution. In the section 4,  the authors  give a necessary and sufficient condition for every matrix in $SL_n(\mathbb{\mathbb{C}})$  to be written as product of commutators of one involution and one skew-involution. Also, we give conditions for what every matrix in $SL_n(\mathbb{\mathbb{C}})$ to be written as products of commutators into matrices of order finite and skew-order finite.
Here we have the following results:

 \begin{theorem}\label{thequivalent} 
Let $A\in SL_n(\mathbb{\mathbb{C}})$. Then $A$ is product of commutators of one involution and one skew-involution if and only if there is a matrix $B\in SL_n(\mathbb{\mathbb{C}})$ such that $A=-B^2$ and $B$ is similar to $-B^{-1}$.
\end{theorem}

And in the general case:

\begin{theorem}\label{casek}
Let $A\in SL_n(\mathbb{C})$. If there is a matrix $B\in SL_n(\mathbb{C})$ such that $A=B^k$ and $B$ is the product of two matrices of order $k$ then  $A$ is product of $2k-3$ commutators of  elements of order $k$.  
\end{theorem}
Also,
\begin{theorem}\label{thskeworderk}
Let $A\in SL_n(\mathbb{\mathbb{C}})$. If there is a matrix $B\in SL_n(\mathbb{\mathbb{C}})$ such that $A=-B^{2k}$ and $B$ is the product of one skew order $2k$ matrix and one matrix of order $2k$ then  $A$ is product of $4k-3$ commutators of these elements. 
\end{theorem}

Considering the results obtained in Gargate \cite{Gargate}, we have the following 

 \begin{theorem}\label{th4}
All element in $SL_n(\mathbb{C})$ can be written as a product of at most

\begin{itemize}
     \item [1.)] Two commutators of involutions in $\pm UT_{n}^{(2)}(\mathbb{C})$.
     \item[2.)] Two commutators of  skew-involutions and  involutions both in $\pm UT_n^{(2)}(\mathbb{C})$.
     \item[3.)] $4k-6$ commutators of matrices of order $k$ in $UT_n^{(k)}(\mathbb{C})$.
     \item[4.)] $8k-6$ commutators of matrices of skew order $2k$ and matrices of order finite $2k$ in $\pm UT_n^{(2k)}(\mathbb{C})$. 
 \end{itemize}

\end{theorem}

Also we consider $GL_{VK}(\infty, \mathbb{C})$ the Vershik-Kerov group and we have the following result:

\begin{theorem}\label{th3} Assume that $R=\mathbb{C}$ is a complex field or the real number field. Then every element of the
group $SL_{VK}(\infty,m, \mathbb{C})$ can be expressed as a product of at most \begin{itemize}
     \item [1.)] Two commutators of involutions.
     \item[2.)] Two commutators of  skew-involutions and  involutions \item[3.)] $4k-6$ commutators of matrices of order $k$.
     \item[4.)] $8k-6$ commutators of matrices of skew order $2k$ and matrices of order finite $2k$. 
 \end{itemize} in
$GL_{VK}(\infty,m, \mathbb{C})$.
\end{theorem}
 \section{Preliminaries}
 
Let $R$ be an associative ring with identity $1$. Denote by $T_n(R)$ and $T_{\infty}(R)$ the groups of $n \times n$ and infinite upper triangular matrices over a ring $R$ and denote by $UT_n(R)$ and $UT_{\infty}(R)$ the subgroups of $T_n(R)$ and $T_{\infty}(R)$, respectively, whose entries on the main diagonal are equal to unity $1$. 

For $m\geq 1$ we denote the subgroups $T_n(m,R)$ and  $T_{\infty}(m,R)$ of $T_n(R)$ and $T_{\infty}(R)$,  respectively, containing exactly all those matrices which have zero entries on the first $m-1$ super diagonals. Analogously we denote by $UT_n(m,R)$ and $UT_{\infty}(m,R)$ the subgroups of $T_n(m,R)$ and $T_{\infty}(m,R)$ respectively whose entries on the $m$th super diagonal are equal to unity $1$.

Also, define the subgroups, for any $s\in \mathbb{N}$ and $m\geq 1$:
 $$\pm UT^{(s)}_{\infty}(R)=\{g\in T_{\infty}(R), \ g_{ii}^{s}=\pm 1\},$$
 $$\pm UT^{(s)}_{\infty}(m,R)=\{g\in T_{\infty}(m,R),\ g_{ii}^s=\pm 1\}$$
$$\pm D^{(s)}_{\infty}(R)=\{g\in UT^{(s)}_{\infty}(R), \ g_{ij}=0, \ if \ i\neq j\},$$
$$\pm D^{(s)}_{\infty}(m,R)=\{g\in UT^{(s)}_{\infty}(m,R), \ g_{ij}=0, \ if \ i\neq j\}.$$

Consider $k\geq 3$, a matrix $A\in M_n(R)$ (or in $T_{\infty}(R)$) is called an
 involution if $A^2=I$, a skew-involution if $A^2=-I$, a finite order $k$ if $A^k=I$ and a skew finite order $k$ if $A^k=-I,$
where $I=I_n$ is the identity matrix in $M_n(R)$ (or $I=I_{\infty}$ in $T_{\infty}(R)$, respectively).

Denote by $E_{ij}$ the finite or infinite matrix with a unique nonzero entry equal to $1$ in the position $(i,j)$, so  $A=\sum_{1\leq i\leq j \leq n}a_{i,j}E_{ij}$ ($A=\sum_{i,j\in \mathbb{N}}a_{ij}E_{i,j}$) is the $n\times n$ (infinite $\mathbb{N}\times \mathbb{N}$) matrix with $a_{ij}$ in the position $(i,j)$. Denote by $[\alpha,\beta]=\alpha\beta\alpha^{-1}\beta^{-1}$ the commutator of two elements $\alpha$ and $\beta$ of a group $G$.

 Denote by $J_m(\infty,R)$ the set of all infinite matrices in $T_{\infty}(R)$ in which all the entries outside of the $m$th  super diagonal equal $0$. Let $A\in UT_{\infty}(m,R)$ and denote by $J_m(A)$ the matrix of $J_m(\infty,R)$ that has the same entries in the  $m$th  super diagonal as $A$. Denote by $Z$ the center of the ring $R$, and by $D_{\infty}(Z)$ the subring of all diagonal infinite matrices with entries in $Z$. 
 
 We say that $A\in UT_{\infty}(m,R)$ is $m$-coherent if there is a sequence $\{D_i\}_{i\geq 0}$ of elements of $D_{\infty}(Z)$ such that  
 $$A=\sum^{\infty}_{i=0}D_iJ_m(A)^i.$$ 
 If $D_0=D_1=I_{\infty}$ the the sequence is called normalized. If $m=1$ the matrix is called coherent.
 
 We will denote by $(*,*,*,\cdots)_m$ the $m$th super diagonal of a matrix $A \in T_{\infty}(m,R)$, for $m\geq 1$.
 
 \section{Expressing Matrix into Products of Commutators}
 In this section we assume that $2$, $k$ or $2k$ are invertible elements in $R$, this the according to each case we will study. Also we will can assume that $i\in R$. 
 
 \begin{remark}
  Let $G=M_n(R)$ or $T_{\infty}(R)$, then
  \begin{itemize}
      \item [1.)] If $\alpha \in G$ is a product of $r$ involution (or  skew-involutions and involutions, or elements of order $k$, or skew-finite order $2k$ and of order $2k$), then for every $\beta \in G$ the conjugate $\beta \alpha \beta^{-1}$ is a product of $r$ involutions (or  skew-involutions and involutions, or elements of order $k$, or skew-finite order $2k$ and of order $2k$, respectively).
      \item[2.)] If $\alpha$ is a product of $r$ commutators of involutions (or  skew-involutions and involutions, or elements of order $k$, or skew-finite order $2k$ and of order $2k$) then for every $\beta \in G$ the conjugate $\beta \alpha \beta^{-1}$ is a product of $r$ commutators of involutions (or  skew-involutions and involutions, or elements of order $k$, or skew-finite order $2k$ and of order $2k$, respectively) as well.
  \end{itemize}
 \end{remark}
 The following results we adapted from Hou \cite{Hou}.
\begin{remark}\label{remark1} If $A\in UT_{\infty}(m,R)$ is $m$-coherent then, for $r\geq 2$, $A^r$ is $m$-coherent as well.
 \end{remark}
 \begin{proof}
  If $D=diag(a_1,a_2,a_3,\cdots)\in D(\infty,Z)$ define $S(D)=diag(a_2,a_3,a_4,\cdots)\in D(\infty,Z)$ and in this case we have that
  
  $$\forall J\in J_m(\infty,R), \ JD=S^m(D)J,$$
  then
 $$\begin{array}{rcl}
     A^2 & =&\displaystyle\sum^{\infty}_{k=0}\sum^{\infty}_{i=0}D_kJ_m(A)^k D_iJ_m(A)^i \\
      & =& \displaystyle\sum^{\infty}_{k=0}\sum^{\infty}_{i=0}D_kS^{km}(D_i)J^{k+i}\\
      & =&\displaystyle\sum^{\infty}_{k=0}\left(\sum^{k}_{i=0}D_{i}S^{im}(D_{k-i})\right)J(A)^k,
 \end{array}
 $$
so, $A^2$ is $m$-coherent. For the case $k\geq 3$ follows similarly from Gargate in \cite{Gargate}.
 \end{proof}

\begin{lemma}\label{lemma1} Let $J\in J_1(\infty,R)$, then there exists a coherent matrix $A\in UT_{\infty}(R)$ such that $J(A)=J$ and $A$ 
is the commutator of one skew-involution and one involution in $\pm UT_{\infty}(R)$. 
\end{lemma}
\begin{proof}
 Suppose that
 $$J=\displaystyle \sum^{\infty}_{i=1} a_{i,i+1}E_{i,i+1},$$
and consider
$$B=\left[\begin{array}{ccccc} -i & & & & \\
  & i & - \frac{1}{2}i a_{23} & & \\ & & -i &  &  \\ & & & i & -\frac{1}{2}i a_{45} \\ & & & &  \ddots \end{array}\right] \ \ and \  \ C=\left[\begin{array}{ccccc} 1 & \frac{1}{2}a
 _{12}& & & \\
  & -1 &  & & \\ & & 1 & \frac{1}{2}a_{34} &  \\ & & & -1 &  \\ & & & &  \ddots \end{array}\right],$$
where $i$ is the imaginary number.

We can observe that $B\in \pm UT^{(2)}_{\infty}(R)$ is a skew-involution  and $C \in \pm UT^{(2)}_{\infty}(R)$ is an involution . Also,  define $A=[B,C]=BCB^{-1}C^{-1}=-(BC)^2$, is not difficult verify that $J(A)=J$. Now, observe that  $BC$ is coherent, because
 $$BC= -i I_{\infty} -\frac{i}{2}\sum^{\infty}_{i=1}a_{i,i+1}E_{i,i+1} - \frac{1}{4}\sum^{\infty}_{i=1}a_{2i,2i+1}a_{2i+1,2i+2}E_{2i,2i+2}=\sum^2_{k=0}D_kJ(BC)^k$$
 where $D_0=-iI_{\infty}, D_1=I_{\infty}$ and $D_2=\sum^{\infty}_{i=1}E_{2i,2i}$ and by Remark \ref{remark1} we conclude that $A$ is coherent. 
 \end{proof}
 And, for the group $UT_{\infty}(m,R)$, we have the following Lemma:
 
 \begin{lemma}
 Let $J\in J_m(\infty,R)$, then there exists a $m$-coherent matrix $A\in UT_{\infty}(m,R)$ such that $J_m(A)=J$ and $A$ is the commutator of 
 
 \begin{itemize}
     \item [1.)] Two involutions in $\pm UT_{\infty}^{(2)}(m,R)$.
     \item[2.)] One skew-involution and one involution both in $\pm UT_{\infty}^{(2)}(m,R)$.
     \item[3.)] Matrices of order $k$ in $UT_{\infty}^{(k)}(m,R)$.
     \item[4.)] One skew finite order $2k$ and one matrix of order finite $2k$ in $\pm UT_{\infty}^{(2k)}(m,R)$. 
 \end{itemize}
\end{lemma} 
 
 \begin{proof}\ \\ \
 1.)  Consider 
  $$J=\sum^{\infty}_{i=1}a_{i,m+i}E_{i,m+i}$$
 and define $B,C\in \pm UT_{\infty}(m,R)$ as

$$B=\left[\begin{array}{cccccccc}
1 & 0 & \cdots & 0 & 0 & 0 & 0 & \cdots \\
0& 1 & \cdots & 0 & \frac{1}{2}a_{2,m+2} & 0 & 0 & \cdots  \\
0 & 0 & \cdots & 0 & 0 & 0 & 0 & \cdots \\
0 & 0 & \cdots & 0 & 0 & 0 & \frac{1}{2}a_{4,m+4} & \cdots \\
\vdots & & & \vdots & & & &  \\
0 & 0 & \cdots & -1 & 0 & 0 & 0 & \cdots \\
0 & 0 & \cdots & 0 & -1 & 0 & 0 & \cdots \\
0 & 0 & \cdots & 0 & 0 & -1 & 0 & \cdots \\
\vdots & & &  \ddots &  & & & \ddots 
\end{array}\right]$$
where 
$$diag(B)=(\underbrace{1,1,1,\cdots, 1}_{m-times}, \underbrace{-1,-1,\cdots, -1}_{m-times}, 1, 1, \cdots ),$$ and $J_m(B)$ the corresponding matrix with entries in the $m$th super diagonal equal to $(0,\frac{1}{2}a_{2,m+2},0,\frac{1}{2}a_{4,m+4},0,\cdots)_m$ and

$$C=\left[\begin{array}{cccccccc}
1 & 0 & \cdots & \frac{1}{2}a_{1,m+1} & 0 & 0 & 0 & \cdots \\
0& 1 & \cdots & 0 & 0 & 0 & 0 & \cdots  \\
0 & 0 & \cdots & 0 & 0 & \frac{1}{2}a_{3,m+3} & 0 & \cdots \\
0 & 0 & \cdots & 0 & 0 & 0 & 0 & \cdots \\
\vdots & & & \vdots & & & &  \\
0 & 0 & \cdots & -1 & 0 & 0 & 0 & \cdots \\
0 & 0 & \cdots & 0 & -1 & 0 & 0 & \cdots \\
0 & 0 & \cdots & 0 & 0 & -1 & 0 & \cdots \\
\vdots & & &  \ddots &  & & & \ddots 
\end{array}\right]$$
with $diag(C)=diag(B)$ and $J_m(C)$ the corresponding matrix with entries in the $m$th super diagonal equals to $(\frac{1}{2}a_{1,m+1},0,\frac{1}{2}a_{3,m+3},0,\frac{1}{2}a_{5,m+5},0,\cdots)_m$. Is not difficult to proved that $B$ and $C$ are involutions, and $J_m(BCB^{-1}C^{-1})=J$.

In order to proved that $BC$ is $m$-coherent we can observe two case:
\begin{itemize}
    \item If $m$ is an odd number then
$$BC=I_{\infty}+\frac{1}{2}\sum^{\infty}_{i=1}a_{i,i+m}E_{i,i+m},$$
and, in this case we consider $D_0=D_1=I_{\infty}$. 
\item If $m$ an even number we have
$$BC=I_{\infty}+\frac{1}{2}\sum^{\infty}_{i=1}a_{i,i+m}E_{i,i+m}+\frac{1}{4}\sum^{\infty}_{i=1}a_{2i,2i+m}a_{2i+m,2i+2m}E_{2i,2i+2m},$$
and  we consider $D_0=D_1=I_{\infty}$ and $D_2=\sum^{\infty}_{i=1}E_{2i,2i}.$
\end{itemize}
So we proof that $BC$ is $m$-coherent.
\\
\newline
2.) We can consider the same form of matrix $B$ and $C$ but with some different entries, for instance, in $B$ 
$$diag(B)=(\underbrace{-i,-i,\cdots,-i}_{m-times},\underbrace{i,i,\cdots,i}_{m-times},-i,-i,\cdots),$$
where $i$ is the imaginary number and the corresponding matrix $J_m(B)$ with entries in the $m$th super diagonal equals to  $(0,-\frac{1}{2}ia_{23},0,-\frac{1}{2}ia_{45},0,\cdots)_m,$
also
$$diag(C)=(\underbrace{1,1,\cdots,1}_{m-times},\underbrace{-1,-1,\cdots,-1}_{m-times},1,1,\cdots),$$
and the corresponding matrix $J_m(C)$ with entries in the $m$th super diagonal equals to $(\frac{1}{2}a_{12},0,\frac{1}{2}a_{34},0,\frac{1}{2}a_{56},\cdots)_m.$

Observe that $B$ is skew-involution and $C$ is an involution and the proof that $BC$ is $m$-coherent follows similarly from  item (1) above.
\\
\newline
3.) For $k> 2$, consider $\omega$ an $k$th root of unity in $R$. Then we consider the matrix $B$ and $C$ of the item (1) but with diagonal
$$diag(B)=(\underbrace{1,1,\cdots,1}_{m-times},\underbrace{\omega,\omega,\cdots,\omega}_{m-times},1,1,\cdots),$$
and the corresponding $J_m(B)$ with entries in the $m$th super diagonal equals to $(0,\frac{1}{k}a_{2,m+2},0,\frac{1}{k}a_{4,m+4},0,\frac{1}{k}a_{6,m+6},\cdots)_m,$
also

$$diag(C)=(\underbrace{1,1,\cdots,1}_{m-times},\underbrace{\omega^{-1},\omega^{-1},\cdots,\omega^{-1}}_{m-times},1,1,\cdots),$$ and the entries of $J_m(C)$ in the $m$th super diagonal equals to $$(\frac{1}{k}a_{1,m+1},0,\frac{1}{k}a_{3,m+3},0,\frac{1}{k}a_{5,m+5},\cdots)_m.$$

We can observe that $B^k=C^k=I$ and similarly to Lemma 3,5 in Gargate \cite{Gargate} we define $A=(BC)^k$ that is product of $2k-3$ commutators.
\\
\newline
4.) For $k\geq 2$ all solutions of the equation $X^{2k}=1$ are of the form $(i\omega)^j$ where $i$ is the imaginary number, $\omega$ be an $k$th root of unity such that $i,\omega \in R$  and $j=0,1,2,\cdots, 2k-1$.
In this case, we consider the above matrices $B$ and $C$ but with diagonal entries
$$diag(B)=(\underbrace{-i,-i,\cdots,-i}_{m-times},\underbrace{i\omega,i\omega,\cdots,i\omega}_{m-times},-i,-i,\cdots),$$
and the corresponding $J_m(B)$ with entries in the $m$th super diagonal equals to $(0,-\frac{1}{2k}ia_{2,m+2},0,-\frac{1}{2k}ia_{4,m+4},0,\frac{1}{2k}ia_{6,m+6},\cdots)_m,$
also

$$diag(C)=(\underbrace{1,1,\cdots,1}_{m-times},\underbrace{\omega^{-1},\omega^{-1},\cdots,\omega^{-1}}_{m-times},1,1,\cdots),$$
and the corresponding $J_m(C)$ with entries in the $m$th super diagonal equals to $(\frac{1}{2k}a_{1,m+1},0,\frac{1}{2k}a_{3,m+3},0,\frac{1}{2k}a_{5,m+5},\cdots)_m.$

Here we can observe that $B^{2k}=-I$ and $C^{2k}=I$. From the Lemma 3.6 in Gargate \cite{Gargate} we can observe that, if $B^{2k}=-I$ in the proof of this Lemma, we obtain that $(BC)^{2k}=-F_{2k}(B,C)$ where $F_{2k}(B,C)$ is the product of $4k-3$ commutators whose entries are powers of $B$ and $C$. In this case we define $A=-(BC)^{2k}$ and obtain the result.
 \end{proof}

 \begin{lemma}\label{lemaprincipal} Let $A,B$ be $m$-coherent matrices of $UT_{\infty}(m,R)$ such that $J_m(A)=J_m(B)$. Then $A$ and $B$ are conjugated in the group $UT_{\infty}(m,R)$.
\end{lemma}

\begin{proof}
Consider $J=J_m(A)=J_m(B)$ and let $(X_n)_{n\geq 0}$ be a sequence of elements of $D_{\infty}(R)$ such that $X_0=I_{\infty}$. Suposse that $X=\sum^{\infty}_{k=0}X_kJ^k\in UT_{\infty}(m,R)$.
Now, choose two normalized sequences $(D_k)_{k\geq 0}$ and $(D^{\prime}_k)_{k \geq 0} $ in $D_{\infty}(R)$ such that 
$A=\sum^{\infty}_{k=0}D_kJ^k$ and $B=\sum^{\infty}_{k=0}D^{\prime}_kJ^k$, and suppose that $AX=XB$, then

$$AX=\sum^{\infty}_{k=0}\left(\sum^k_{i=0}D_iS^{im}(X_{k-i})\right)J^k \  \ and \ \ 
XB=\sum^{\infty}_{k=0}\left(\sum^k_{i=0}X_iS^{im}(D^{\prime}_{k-i})\right)J^k.$$
so to have the result it is enough to prove that $\forall k\geq 2:$

$$\sum^k_{i=0}D_iS^{im}(X_{k-i})=\sum^k_{i=0}X_iS^{im}(D^{\prime}_{k-i}).$$
Check this condition for $k=0,1$. If $k\geq 2$ then we have

$$X_k+S^m(X_{k-1})+\sum^k_{i=2}D_iS^{im}(X_{k-i})=X_k+ X_{k-1}S^{(k-1)m}(D^{\prime}_1)+\sum^{k-2}_{i=0}X_iS^{im}(D^{\prime}_{k-i}) $$
but $D^{\prime}_1=I_{\infty}$, then rewrite the last equality, for all $k\geq 2$
$$S^m(X_k)-X_k=\sum^{k-1}_{i=0}X_iS^{im}(D^{\prime}_{k+1-i})-\sum^{k+1}_{i=2}D_iS^{im}(X_{k+1-i}), $$
and consider the first $m$th super diagonal entries in the each $X_k$ be zero for all $k\geq 1$, then we can define a sequence $(X_k)_{k\geq 0}$ inductively and with this we proof that $AX=XB$. 
\end{proof}
From Lemmas \ref{lemma1} and  \ref{lemaprincipal}, we obtain the follow result.

 \begin{corollary}\label{coro1}  Let $A \in  UT_{\infty}(R)$ whose entries except in the main diagonal and the first super diagonal are all equal to zero. Then $A$ is a commutator of one skew-involution and one involution.
 \end{corollary}
 \begin{proof}
  We know that $A$ is coherent. Consider $J=J_1(A)$ and by Lemma \ref{lemma1} there is $T \in UT_{\infty}(R)$ coherent such that $J=J_1(T)$ and $T=[S,P]$ with $S$ an skew-involution and $T$ and involution respectively. Finally, by the Lemma \ref{lemaprincipal}, we concluded that $A$ e $T$ are conjugated.
   \end{proof}
 And, for the subspace $UT_{\infty}(m,R)$ we have  similar results

\begin{corollary}\label{coro2}
Assume that $R$ is an associative ring with identity $1$ and $2$ (also $k$ or $2k$) is an invertible element of $R$. Then for every $A\in UT_{\infty}(m,R)$ (or $UT_n(m,R)$), whose entries except in the main diagonal and the $m$th super diagonal are equall to zero, is a commutator of 

 \begin{itemize}
     \item [1.)] Two involutions in $\pm UT_{\infty}^{(2)}(m,R)$.
     \item[2.)] One skew-involution and one involution both in $\pm UT_{\infty}^{(2)}(m,R)$.
     \item[3.)] Matrices of order $k$ in $UT_{\infty}^{(k)}(m,R)$.
     \item[4.)] One skew finite order $2k$ and one matrix of order finite $2k$ in $\pm UT_{\infty}^{(2k)}(m,R)$. 
 \end{itemize}

\end{corollary}
\begin{proof}
Similar to the Corollary \ref{coro1} and Corollary 3.7 in \cite{Gargate}.
\end{proof}
 
For the group $UT_{\infty}(m,R)$ we have the following Lemma.

\begin{lemma}\label{matrix} Let $R$ be an associative ring with identity $1$ and let $n\in \mathbb{N}$.
\begin{itemize}
\item [1.)] If $A,B\in UT_{n}(m,R)$ such that $a_{i,m+i}=b_{i,m+i}=1$ for all $1\leq i\leq n-1,$ then $A$ and $B$ are conjugated in $UT_n(m,R)$.
\item[2.)]  If $A,B\in UT_{\infty}(m,R)$ such that $a_{i,m+i}=b_{i,m+i}=1$ for all $1\leq i$, then $A$ and $B$ are conjugated in $UT_{\infty}(m,R)$.
\end{itemize}
\end{lemma}

\begin{proof}\ For the case $m=1$ see the Lemma 2.6 in Hou \cite{Hou}. Now consider $m>1$. 
\begin{itemize}
\item[1.)]
Consider $A=(a_{ij})\in UT_{n}(m,R)$ and $$J=\left(\begin{array}{cccccccccc} 
1 & 0 &\cdots&0 & 1&  &   & & \\
& \ddots & &\ddots& &\ddots & && \\ 
&  & 1 & 0&\cdots &0 &1 & &   \\
& &  &\ddots & & \ddots  & &\ddots&\\
 & & & & 1 & 0&\cdots &0 &1  \\ 

\end{array}\right),$$
where all blank entries are equal to $0$. We need only to prove that $A$ is conjugated to the matrix $J$, for this we can constructed a matrix $X=(x_{ij})\in UT_{n}(R)$ such that $X^{-1}AX=J$ or $AX=XJ$. 

Let $X$ be the matrix

$$X=\left(\begin{array}{cccccccccc}
1 & x_{12} &&&\cdots&    &\multicolumn{1}{c|}{x_{1,n-m}}  &  &&\\
 & \ddots &&   & &    & \multicolumn{1}{c|}{\vdots} &  &X'&\\
 &  &1&&\cdots&     &\multicolumn{1}{c|}{x_{m,n-m}}  &  &&\\ \cline{8-10}
 &  &&\ddots&&    & \vdots  &  &&\\
 &  &&& &    & &  &&\\
 &  &&& & 1   &  &  &&\\
 &  && &    && 1 &  &&\vdots\\
 &  &&& &    &  & \ddots &&\\
 &  &&& &    &  &  &1&x_{n-1,n}\\
 &  && &    &  &  &&&1
\end{array}\right)$$
where $X'$ is a matrix of order $m\times m$ with arbitrary entries and the other entries of the matrix $X$ are related as  follows
 
\begin{itemize}
    \item [(a)]  For the first super diagonal entries of
$X$ where $j=i+1$ we choose  
    $$x_{ij}= x_{i+m,j+m}+a_{i,j+m}
    $$
    
    \item[(b)] For $j=i+k$ with $k\geq 2$, the following super diagonal entries are obtained from 
    $$x_{i,j}=x_{i+m,j+m}+a_{i,j+m} +\sum_{s=1}^k a_{i,j+m-(k-s)} \cdot x_{i+m+s,j+m}.   $$
\end{itemize}

\item[2.)]  If $A=(a_{ij})\in UT_{\infty}(m,R)$ then $A$ is conjugated to the matrix $J\in UT_{\infty}(m,R)$ as in $(1)$ and here constructed a matrix $X=(x_{ij})\in UT_{\infty}(m,R)$  such that $AX=XJ$. The relations of $(x_{ij})$ in the super diagonal entries are the same as those given in $(1)$ with the difference that the matrix $X'$ disappears.
\end{itemize}   
Thus the Lemma \ref{matrix} is proved.
\end{proof}

 From Corollary \ref{coro2} and Lemma \ref{matrix} we have
 
\begin{corollary}\label{coro22}
Let $R$ be an associative ring with unity $1$ and that $2$ ($k$ or $2k$ respectively) is an invertible element in $R$. Every matrix in $UT_{\infty}(m,R)$ whose entries in the diagonal and the $m$th super diagonal are all equal to the identity $1$, is an commutator of 
\begin{itemize}
     \item [1.)] Two involutions in $\pm UT_{\infty}^{(2)}(m,R)$.
     \item[2.)] One skew-involution and one involution both in $\pm UT_{\infty}^{(2)}(m,R)$.
     \item[3.)] Two matrices of order $k$ in $UT_{\infty}^{(k)}(m,R)$.
     \item[4.)] One skew finite order $2k$ and one matrix of order finite $2k$ in $\pm UT_{\infty}^{(2k)}(m,R)$. 
 \end{itemize}
'
\end{corollary}
\begin{proof}
Follows similar from Hou in \cite{Hou} and Gargate in \cite{Gargate}.
\end{proof}

 Finally, we  proved the Main Theorem:

 \begin{proof}[Proof of Theorem \ref{th1}]
 Consider $A\in UT_{\infty}(m,R)$ then we can write
 $$A=I_{\infty}+\sum^{\infty}_{i=1}\sum^{\infty}_{j=m+i}a_{i,j}E_{i,j}$$
 and consider 
 $$A_1=I_{\infty}+\sum^{\infty}_{i=1}(a_{i,m+i}-1)E_{i,m+i} \in UT_{\infty}(m,R),$$
then, by corollary \ref{coro1}, $A_1$ is a commutators as desired. Observe that 
  $$A_2=A^{-1}_1 A\in UT_{\infty}(m,R)$$ 
  is an matrix whose entries in the diagonal and the $m$th super diagonal are all equal to $1$, then by Corollary \ref{coro2} $A_2$ is also commutator as desired. So $A$ is product of  commutators according to each case. 
 \end{proof}

\section{Case $R=\mathbb{C}$ the complex field and the Group Vershik-Kerov}

In this section, we consider the case $R=\mathbb{C}$ a complex field  and the group $SL_n(\mathbb{C})$. We proved the Theorem \ref{thequivalent}:

\begin{proof}[Proof of Theorem \ref{thequivalent}]  \

$(\Rightarrow)$ If $A=[S,T]$ with $S$ one involution and $T$ one skew-involution then $A=-(ST)^2$. Consider $B=ST$ then  $B^2=-A$ and $-B^{-1}=TS$. Hence $$B=ST= S(TS)S^{-1}=S(-B^{-1})S^{-1},$$
then, $B$ is similar to $-B^{-1}$.

$(\Leftarrow)$ By hypothesis, if $B\in SL_n(\mathbb{C})$ is such that $A=-B^2$ and $B$ is similar to $-B^{-1}$ then, by the Theorem 5 in \cite{Paras}, $B$ is product of one involution $S$ and one skew-involution $T$. Then we have
$$A=-B^{2}=-(ST)^2=-STST=-ST{S^{-1}}(-T^{-1})=STS^{-1}T^{-1}=[S,T].$$
\end{proof}
An immediate consequence  is the following corollary:

\begin{corollary}\label{corolinvol}Let $A\in SL_n(\mathbb{C})$. Then 
\begin{itemize}
    \item [1.)] $A$ is a product of commutators of involutions if there is a matrix $B\in SL_n(\mathbb{C})$ such that $A=B^2$ and $B$ is product of two involutions.
    \item[2.)] $A$ is a product of commutators of one skew-involution and one involution if there is a matrix $B\in SL_n(\mathbb{C})$ such that $A=-B^2$ and $B$ is a product of one skew-involutions and one involution.
 \end{itemize}
\end{corollary}
\begin{proof}
Follows immediately from proof of Theorem \ref{thequivalent}.
\end{proof}
Next, we proved the Theorem \ref{casek}.

\begin{proof}[Proof of Theorem \ref{casek}] In this case, we can rewrite $A=B^k=(CD)^k$ with $C^k=D^k=I$ and the result follows from Lemma 3.6 in Gargate \cite{Gargate}.
\end{proof}

And, in the general case we have as a result the Theorem \ref{thskeworderk}:

\begin{proof} [Proof of Theorem \ref{thskeworderk}]
Follows immediately from Lemma 3.6 in Gargate \cite{Gargate}.
\end{proof}

 Then, in $SL_n(\mathbb{C})$ we proved the Theorem \ref{th4}:

\begin{proof}[Proof of Theorem \ref{th4}]
Consider $A\in SL_n(\mathbb{\mathbb{C}}) $ not a scalar matrix, then by Theorem 1 in \cite{Sourour}, we can find a lower-triangular matrix $L$ and a upper-triangular matrix $U$ such that $A$ is similar to $LU$, and both $L$ and $U$ are unipotent. By the Theorem \ref{th1} it follows that each one of the matrices $L$ and $U$ is a product of commutators as desired. 
 For the scalar case $A=\alpha I$ with $det(A)=1$ it suffices to consider the case when $n$ is exactly the order of $\alpha$ (see \cite{Hou}).
 We using the techniques of \cite{Grunenfelder} for the proof in each case. Observe that, if $n$ is even we have

$$\alpha I=\left[\begin{array}{cccccc}
\alpha &             & & & & \\
       & \alpha^{-1} & & & & \\
  & & \alpha^3 & &  & \\
  & & &  \ddots & & \\
   & & &         &  \alpha^{n-1} & \\
    & & &        &  & \alpha^{-n+1}\end{array}\right] \left[\begin{array}{cccccc}
1 &             & & & & \\
       & \alpha^{2} & & & & \\
  & & \alpha^{-2} & &  & \\
  & & &  \ddots & & \\
   & & &         &  \alpha^{-n+2} & \\
    & & &        &  & 1\end{array}\right],$$
and if $n$ is odd
 $$\alpha I=\left[\begin{array}{cccccc}
\alpha &             & & & & \\
       & \alpha^{-1} & & & & \\
  & & \alpha^3 & &  & \\
  & & &  \ddots & & \\
   & & &         &  \alpha^{-n+2} & \\
    & & &        &  & 1\end{array}\right] \left[\begin{array}{cccccc}
1 &             & & & & \\
       & \alpha^{2} & & & & \\
  & & \alpha^{-2} & &  & \\
  & & &  \ddots & & \\
   & & &         &  \alpha^{n-1} & \\
    & & &        &  & \alpha^{-n+1}\end{array}\right].$$

Thus, if $n$ is even then denote by $$S=diag(\alpha,\alpha^{-1},\alpha^3,\cdots,\alpha^{n-1},\alpha^{-n+1})\ \  \text{and}\ \ T=diag(1,\alpha^2,\alpha^{-2},\cdots, \alpha^{-n+2},1),$$
so $\alpha I=ST.$ The objective is analyse, in each case, the decomposition in products of the matrix 
 $$\left[\begin{array}{cc}a & 0 \\ 0 & a^{-1}\end{array}\right]$$
with $a\in \mathbb{C}$, $a\notin \{0,1,-1\}$.
\\
\newline
1.) For involutions, it's enough to find two matrices $B$ and $C$ such that $\alpha I=B^2C^2$ and both are products of involutions. For this case consider $$B=diag(\alpha^{1/2},\alpha^{-1/2},\alpha^{3/2},\cdots, \alpha^{(-n+2)/2}, 1),$$
and observe that $B^2=S$. Here
$$\left[\begin{array}{cc}a & 0 \\ 0 & a^{-1}\end{array}\right]=J_1(a)\cdot J_2(a),$$
with $J_1(a)=\left[\begin{array}{cc}0 & ay \\ \displaystyle\frac{1}{ay} & 0\end{array}\right]$ and $J_2(a)=\left[\begin{array}{cc}0 & y \\ \displaystyle\frac{1}{y} & 0\end{array}\right]$,  $y \in \mathbb{C}\setminus\{0\}$ and both are involutions. Then it follows that $B$ is product of two involutions, so $S$ is a commutator of these  two involutions. Similarly we can proof the same result for $T$. Observe that for the case $n$ odd number the proof is similar, so we can conclude that $\alpha I$ is a product of two commutators of involutions. 
For other proof of this case see Hou in \cite{Hou}.
\\
\newline
2.) For this case we can see the following decomposition, if $n$ is even
$$\alpha I=  diag(-\alpha,-\alpha^{-1},-\alpha^3,\cdots,-\alpha^{n-1},-\alpha^{-n+1})\times $$ $$ \ \ \ \ \  diag(-1,-\alpha^2,-\alpha^{-2},\cdots,-\alpha^{-n+2},-1)$$
and if $n$ is odd
$$\alpha I= - diag(-\alpha,-\alpha^{-1},-\alpha^3,\cdots,-\alpha^{-n+2},-1)\times $$ $$ \ \ \ \ \ \ \ diag(-1,-\alpha^2,-\alpha^{-2},\cdots,-\alpha^{n-1},-\alpha^{-n+1}).$$

For $n$ even, consider $S=diag(-\alpha,-\alpha^{-1},-\alpha^3,\cdots,-\alpha^{-n+2},-1)=-B^2$ with $B=diag(i\alpha^{1/2},i\alpha^{-1/2},i\alpha^{3/2},\cdots,i\alpha^{(-n+2)/2},i)$.
And, observe that, for all $a\in \mathbb{C}$ $a\neq 0$,
$$\left[\begin{array}{cc}ia & 0 \\ 0 & ia^{-1}\end{array}\right]=(iJ_1(a))\cdot J_2(a)$$
with $J_1(a)$ and $J_2(a)$ the same matrices in the item (1). Observe that, in this case $i J_1(a)$ is an skew-involution and $J_2(a)$ is an involution. So, $B$ is the product of one skew-involution and one involution, then by the Corollary  \ref{corolinvol} we have that $S$ is a commutator of one skew-involution and one involution. Follows similarly for $T$ and so we can conclude that $\alpha I$ is a product of two commutators of skew-involutions and involutions. For the case $n$ odd number follows immediately.
\\
\newline
3.) Observe that for $a\in \mathbb{C}$, $a\notin \{0,1,-1\}$ then
 $$\left[\begin{array}{cc}a & 0 \\ 0 & a^{-1}\end{array}\right]= J_1(a)\cdot J_2(a)$$
 with
 $$J_1(a)=\left[\begin{array}{cc}\displaystyle\frac{at}{a+1} & a-\left(\displaystyle\frac{at}{a+1}\right)^2 \\
 & \\ -\displaystyle\frac{1}{a} & \displaystyle\frac{t}{a+1}\end{array}\right] \ and \  J_2(a)=\left[\begin{array}{cc}\displaystyle\frac{at}{a+1} & a\left(\displaystyle\frac{at}{a+1}\right)^2-1 \\ 
 &\\ 1 & \displaystyle\frac{t}{a+1}\end{array}\right],$$
where $t=\theta+\theta^{-1}$, $\theta^k=1$, $\theta\neq 1$. Observe that $J_1(a)^k=J_2(a)^k=I$ (see \cite{Grunenfelder}).

Then, if $n$ is even we have 

$$\alpha I=\left[\begin{array}{cccccc}
\alpha &             & & & & \\
       & \alpha^{-1} & & & & \\
  & & \alpha^3 & &  & \\
  & & &  \ddots & & \\
   & & &         &  \alpha^{n-1} & \\
    & & &        &  & \alpha^{-n+1}\end{array}\right] \left[\begin{array}{cccccc}
1 &             & & & & \\
       & \alpha^{2} & & & & \\
  & & \alpha^{-2} & &  & \\
  & & &  \ddots & & \\
   & & &         &  \alpha^{-n+2} & \\
    & & &        &  & 1\end{array}\right],$$
and if $n$ is odd
 
 $$\alpha I=\left[\begin{array}{cccccc}
\alpha &             & & & & \\
       & \alpha^{-1} & & & & \\
  & & \alpha^3 & &  & \\
  & & &  \ddots & & \\
   & & &         &  \alpha^{-n+2} & \\
    & & &        &  & 1\end{array}\right] \left[\begin{array}{cccccc}
1 &             & & & & \\
       & \alpha^{2} & & & & \\
  & & \alpha^{-2} & &  & \\
  & & &  \ddots & & \\
   & & &         &  \alpha^{n-1} & \\
    & & &        &  & \alpha^{-n+1}\end{array}\right].$$
 
 Denote by $F$ and $G$ the first and second matrix that appears in these decompositions. If $n$ is even, for instance, $F=diag(\alpha,\alpha^{-1},\alpha^3,\cdots, \alpha^{n-1},\alpha^{-n+1})$ and let $S=diag(\alpha^{1/k},\alpha^{-1/k},\alpha^{3/k},\cdots, \alpha^{(n-1)/k},\alpha^{(-n+1)/k}).$ Then, each block $2\times 2$ of the form $diag(\alpha^j,\alpha^j)$ is product of $J_1(\alpha^j)$ and $J_2(\alpha^j)$, both of order $k$, so $F=S^k=(J_1\cdot J_2)^k$ with $J_1$ and $J_2$ the respective block matrices. By the Theorem \ref{casek} we have that $F$ is product of $2k-3$ commutators of $J_1$ and $J_2$. Similarly we can obtain the same results for the others matrices. Therefore, by the decomposition, we conclude that $\alpha I_n$ is a product of $4k-6$ commutators. The case $n$ odd number follows immediately.
\\
\newline
4.) For $n$ even, consider
$$\alpha I=-diag(-\alpha,-\alpha^{-1},-\alpha^3,\cdots, - \alpha^{n-1},-\alpha^{-n+1})\times$$ $$ \ \ \ \ \ \ \  \times diag(-1,-\alpha^2,-\alpha^{-2},\cdots,-\alpha^{-n+2},-1), $$
and $S=diag(-\alpha,-\alpha^{-1},-\alpha^3,\cdots, - \alpha^{n-1},-\alpha^{-n+1})=-B^{2k}$ with  $$B=diag(i^{1/k}\alpha^{1/(2k)},i^{1/k}\alpha^{-1/(2k)},\cdots,i^{1/k}\alpha^{(-n+1)/(2k)}),$$ 
where $i$ is the imaginary number.

Observe that
$$\left[\begin{array}{cc}i^{1/k}a & 0 \\ 0 & i^{1/k}a^{-1}\end{array}\right]=(i^{1/k}J_1(a))\cdot J_2(a),$$
and $i^{1/k}J_1(a)$ is a skew order $2k$ matrix. By Theorem \ref{thskeworderk} we conclude that $S$ is a product of $4k-3$ commutators. By the same observations in the above items we conclude the proof. 
\end{proof}
Let $n$ a positive integer  and consider $GL_n(\mathbb{C})$ the general linear group over $\mathbb{C}$. The Vershik-Kerov group  $GL_{VK}(\infty,\mathbb{C})$ is the group consisting of all infinite matrices of the form
\begin{equation}\label{eq111}
\left(\begin{array}{c|c}M_1 & M_2 \\ \hline 0 & M_3 \end{array}\right)
\end{equation}
    where $M_1\in GL_n(\mathbb{C})$ and $M_3\in T_{\infty}(\mathbb{C}).$ 
Also, denote by $GL_{VK}(\infty,m,R)$ the subgroup of $GL_{VK}(\infty,\mathbb{C})$ such that $M_3 \in T_{\infty}(m,R)$ and denote by $SL_{VK}(\infty,m,R)$ the subgroup of $GL_{VK}(\infty,m,R)$ such that $M_1\in SL_n(\mathbb{C})$ and $M_3 \in UT_{\infty}(m,R)$.

We  use the following lemma proof in \cite{Hou}:

\begin{lemma}
Assume that $\mathbb{C}$ is a complex field or the real number field. Let $A\in GL_n(\mathbb{C})$ of which $1$ is no eigenvalue, and let $T$ be an infinite unitriangular matrix. In the Vershik-Kerov group, any matrix of the form 
\begin{equation}\label{eq112}
\left(\begin{array}{c|c}A & B \\ \hline 0 & T \end{array}\right)
\end{equation} is conjugated to \begin{equation}\label{eq113}
\left(\begin{array}{c|c}A & 0 \\ \hline 0 & T \end{array}\right)
\end{equation}
\end{lemma}
 
Then, in this case we shown the Theorem \ref{th3}

\begin{proof}[Proof of Theorem \ref{th3}]  Consider $M\in SL_{VK}(m,\mathbb{C})$ in the form $M=\left(\begin{array}{c|c} M_1 & M_2  \\ \hline 0 & M_3\end{array}\right)$, with $M_1\in SL_n(\mathbb{C})$ and $M_3\in UT_{\infty}(m,\mathbb{C})$. From the proof of Theorem 1.3 in \cite{Hou}, $M$ is conjugated to an infinite matrix of the form $\left(\begin{array}{c|c} A & 0  \\ \hline 0 & T\end{array}\right)$, with $A\in SL_n(\mathbb{C})$ for which $1$ is no eigenvalue and $T \in UT_{\infty}(m,\mathbb{C})$. By the Theorem \ref{th1} and the Theorem \ref{th4}, both are products of  commutators  and we know that the direct sum of $A$ and $T$ is also a product of commutators of elements as desired.
\end{proof}

\end{document}